\documentclass[reqno]{amsart}

\usepackage{amsmath,amsfonts,amssymb}
\usepackage{times}
\usepackage[centertags]{amsmath}
\usepackage{amsfonts}
\usepackage{amssymb}
\usepackage{amsthm}
\usepackage{newlfont}

\newtheorem{thm}{Theorem}
\newtheorem{prop}{Proposition}
\newtheorem{lem}{Lemma}

\newtheorem{definition}{Definition}
\newtheorem{cor}{Corollary}

\begin{document}

\title{Some Generalizations of Hermite-Hadamard type Integral Inequalities and Their Applications}

\author{Muhammad Muddassar}
\address{Department of Mathematics\\University of Engineering and Technology, Taxila. Pakistan}
\email{malik.muddassar@gmail.com}
\author{Muhammad Iqbal Bhatti}
\address{Department of Mathematics\\University of Engineering and Technology, Lahore. Pakistan}
\email{uetzone@hotmail.com}
\date{\today}

\begin{abstract}
In this paper, we establish various inequalities for some differentiable mappings that are linked with the illustrious Hermite- Hadamard integral inequality for mappings whose derivatives are $(h-(\alpha,m))$-convex.The generalized integral inequalities contribute some better estimates than some already presented. The inequalities are then applied to numerical integration and some special means.
\end{abstract}
\maketitle
{\setcounter{section}{0}}
\markboth{\underline{\hspace{3.80in}M. Muddassar and M. I. Bhatti}}
{\underline{\hspace{1pt}Some Generalizations of Hermite-Hadamard type Integral Inequalities and Their Applications\hspace{1.40in}}}\pagestyle{myheadings}
\section{Introduction}\label{Sec1}
Let $f: I\subset\mathbb{R}\rightarrow\mathbb{R}$ be a function defined on the interval $I$ of real numbers. Then $f$ is called convex if
\begin{equation*}
    f(t\,x+(1-t)\,y)\leq\,\,t\,f(x)\,+\,\,(1-t)\,f(y)
\end{equation*}
for all $x,y\in I$ and $t\in [0,1]$. Geometrically, this means that if P, Q and R are three distinct points on graph of $f$ with Q between P and R, then Q is on or below chord PR. There are many results associated with convex functions in the area of inequalities, but one of those is the classical Hermite Hadamard inequality:
\begin{equation}\label{HH}
    f\left(\frac{a+b}{2}\right)\leq\frac{1}{b-a}\int_a^b f(x) dx\leq\frac{f(a)+f(b)}{2}
\end{equation}
for $a,b\in I,$ with $a<b$.\\
In \cite{r5}, H. Hudzik and L. Maligranda considered, among others, the class of functions which are $s-$convex in the first and second sense. This class is defined as follows:
\begin{definition}
A function $f:[0,\infty)\rightarrow\mathbb{R}$ is said to be $s-$convex or $f$ belongs to the class $K_s^i$
if
\begin{equation}\label{e1}
    f(\mu \,x + \nu \,y) \leq\,\, \mu^s\,f(x)\,+\, \nu^s\,f(y)
\end{equation}
holds for all $x,y \in [0,\infty),$ $\mu, \nu \in[0,1]$ and for some fixed $s\in(0,1]$.
\end{definition}
Note that, if $\mu^s + \nu^s=1$, the above class of convex functions is called $s$-convex functions in first sense and represented by $K_s^1$ and if $\mu + \nu=1$ the above class is called $s$-convex in second sense and represented by $K_s^2$.\\
It may be noted that every 1-convex function is convex. In the same paper \cite{r5} H. Hudzik and L. Maligranda discussed a few results connecting with $s-$convex functions in second sense and some new results about Hadamard's inequality for $s-$convex functions are discussed in \cite{r4}, while on the other hand there are many important inequalities connecting with 1-convex (convex) functions \cite{r4}, but one of these is $(\ref{HH}).$\\
In ~\cite{r9}, V.G. Mihesan presented the class of $(\alpha,m)$-convex functions as reproduced below:
\begin{definition}
The function $f:[0,b] \rightarrow \mathbb{R}$ is said to be $(\alpha,m)$-convex, where $(\alpha,m) \in [0,1]^2$, if for every $x,y \in [0,b]$ and $t \in [0,1]$ we have
$$f(tx + m(1 - t)y)  \leq t^\alpha f(x) + m(1 - t^\alpha)f(y)$$
\end{definition}
Note that for $(\alpha,m) \in \{(0,0) ,(\alpha,0), (1,0), (1,m), (1,1), (\alpha,1)\}$ one receives the following classes of functions respectively:  increasing, $\alpha$-starshaped, starshaped, $m$-convex, convex and $\alpha$-convex.\\
Denote by $K_m^\alpha(b)$ the set of all $(\alpha,m)$-convex functions on $[0, b]$ with $f(0) \leq 0$. For recent results and generalizations referring m-convex and  $(\alpha,m)$-convex functions see ~\cite{r1}, ~\cite{r2} and ~\cite{r15}.\\
M. Muddassar et. al., define a new class of convex functions in ~\cite{r10a} named as $s$-$(\alpha, m)$-convex functions as reproduced below
\begin{definition}
A function $f : [0, \infty) \rightarrow [0, \infty)$ is said to be $s$-$(\alpha,m)$-convex function in first sense or f belongs to the class ${K_{m, 1}^{\alpha, s}}$ , if for all $x, y \in [0, \infty)$ and $\mu \in  [0, 1]$, the following inequality holds:
\begin{equation*}
    f(\mu x + (1-\mu) y) \leq \left({\mu^\alpha}^s \right) f(x) + m \left(1-{\mu^\alpha}^s \right) f\left(\frac{y}{m}\right)
\end{equation*}
where $(\alpha,m) \in [0,1]^2$ and for some fixed $s \in (0, 1]$.
\end{definition}
\begin{definition}
A function $f : [0, \infty) \rightarrow [0, \infty)$ is said to be $s$-$(\alpha,m)$-convex function in second sense or f belongs to the class ${K_{m, 2}^{\alpha, s}}$ , if for all $x, y \in [0, \infty)$ and $\mu, \nu \in  [0, 1]$, the following inequality holds:
\begin{equation*}
    f(\mu x + (1-\mu) y) \leq \left(\mu^\alpha \right)^s f(x) + m \left(1-\mu^\alpha \right)^s f\left(\frac{y}{m}\right)
\end{equation*}
where $(\alpha,m) \in [0,1]^2$ and for some fixed $s \in (0, 1]$.
\end{definition}
Note that for $s=1$, we get $K_m^\alpha(I)$ class of convex functions and for $\alpha=1$ and $m=1$, we get $K_s^1(I)$ and $K_s^2(I)$ class of convex functions.\\
In ~\cite{r17}, S. Varo\v{s}anec define the following class of convex functions as reproduced below:
\begin{definition}
Let $h: \mathcal{J} \subseteq \mathbb{R} \rightarrow \mathbb{R}$ be a non-negative function, $h\neq0$. We say that $f: I\subseteq \mathbb{R} \rightarrow \mathbb{R}$ is an $h$-convex function (or that $f$ belongs to the class $SX(h, I)$) if $f$ is non-negative and for all $x, y \in I$, $\mu, \nu \in (0,1)$ and $\mu + \nu =1$, we have
\begin{equation*}
    f(\mu x + \nu y) \leq h(\mu)f(x) + h(\nu)f(y)
\end{equation*}
if the above inequality reversed, then $f$ is said to be $h$-concave (or $f \in SV(h, I)$).
\end{definition}
Evidently, if $h(\mu)= \mu$, then all non-negative convex functions belong to $SX(h, I)$ and all non-negative concave functions belong to $SV(h, I)$; if $h(\mu)= \frac{1}{\mu}$, then $SX(h, I) = Q(I)$; if $h(\mu)= 1$, then $P(I) \subseteq SX(h, I)$; and if $h(\mu)= \mu^s$, where $s \in(0,1]$, then $K_s^2 \subseteq SX(h, I)$.
In ~\cite{r15}, M. E. \"{O}zdemir et. al., define a new class of convex functions as below:
\begin{definition}
Let $h: \mathcal{J} \subseteq \mathbb{R} \rightarrow \mathbb{R}$ be a non-negative function, $h\neq0$. We say that $f: I\subseteq \mathbb{R} \rightarrow \mathbb{R}$ is an $(h-(\alpha,m))$-convex function (or that $f$ belongs to the class $SX((h_(\alpha,m)), I)$) if $f$ is non-negative and for all $x, y \in I$ and $\lambda \in (0,1)$ for $(\alpha,m) \in [0,1]^2$, we have
\begin{equation*}
    f(\lambda x + m (1-\lambda) y) \leq h^\alpha(\lambda)f(x) + m(1-h^\alpha(\lambda))f(y)
\end{equation*}
if the above inequality is reversed, then $f$ is said to be $(h-(\alpha,m),I)$-concave, \emph{i.e.}, $f\in SV(h-(\alpha,m), I)$.
\end{definition}
Evidently, if $h(\lambda)= \lambda$, then all non-negative convex functions belong to $K_m^\alpha(I)$.
In ~\cite{r4} S. S. Dragomir et al. discussed inequalities for differentiable and twice differentiable functions connecting with the H-H Inequality on the basis of the following Lemmas.
\begin{lem}\label{L1}
Let $f:I\subseteq\mathbb{R}\rightarrow\mathbb{R}$ be differentiable function on $I^\circ$ (interior of $I$), $a,b \in I$ with $a < b$. If $f' \in L^{1}([a,b])$, then we have
\begin{eqnarray}\label{e3}
&&\nonumber\!\!\!\!\!\!\!\!\!\!\!\!\!\!\!\!\!\!\!\!f\left(\frac{a+b}{2}\right)-\frac{1}{b-a}\int_a^b f(x) dx=\frac{(b-a)}{4}\!\int_0^1 \!\!(1-t) \left[f'\left(ta+(1-t)\frac{a+b}{2}\right)\right.\\
&&\indent\indent\indent\indent\indent\indent\indent\indent\indent\indent\indent\indent\indent\left.-f'\left(tb+(1-t)\frac{a+b}{2}\right)\right]dt
\end{eqnarray}
\end{lem}
In ~\cite{r3}, Dragomir and Agarwal established the following results connected with the right part of (\ref{e3}) as well as to apply them for some elementary inequalities for real numbers and numerical integration.
\begin{lem}\label{L2}
Let $f:I^\circ \subseteq \mathbb{R} \rightarrow \mathbb{R}$ be differentiable function on $I^\circ,$ $a,b\in I$ with $a<b$. If $f'\in L^{1}[a,b],$ then
\begin{equation}\label{e4}
\frac{f(a)+f(b)}{2}-\frac{1}{b-a}\int_a^b f(x)dx   = \frac{(b-a)^2}{2}\int_0^1 t(1-t)f''(ta+(1-t)b)dt
\end{equation}
\end{lem}
Here We give definition of Beta function of Euler type which will be helpful in our next discussion, which is for $x,y>0$ defined as
$$\beta(x,y)= \frac{\Gamma(x).\Gamma(y)}{\Gamma(x+y)} = \int_0^1\,t^{x-1}\,(1-t)^{y-1}\,\,dt$$
This paper is in the direction of the results discussed in ~\cite{r6} but here we use $(h-(\alpha,m))$-convex functions instead of $s$-convex function. After this introduction, in section \ref{sec2} we found some new integral inequalities of the type of Hermite Hadamard's for generalized convex functions. In section \ref{sec3} we give some new applications of the results from section \ref{sec2} for some special means. The inequalities are then applied to numerical integration in section \ref{sec4}.
\section{Main Results}\label{sec2}
The following theorems were obtained by using the $(h-(\alpha,m))$-convex function.
\begin{thm}\label{T1}
Let $f:I^o\subseteq \mathbb{R}\rightarrow \mathbb{R}$ be a differentiable function on $I^o$ (interior of I), $a,b \in I$ with $a<b$. If $f' \in L^1[a,b]$. If the mapping $|f'|$ is $(h-(\alpha,m))$-convex on $[a,b]$, then
\begin{eqnarray}\label{Te1}
&&\!\!\!\!\!\!\!\!\!\!\!\!\!\!\!\!\!\nonumber\left|f\!\left(\!\frac{a+b}{2}\!\right)\!-\!\frac{1}{b-a}\!\int_a^b\!\! f(x)dx\right|\!\leq\! \frac{(b-a)}{4}\!\left[\!\left\{|f'(a)|+ |f'(b)| +2m\left|f'\!\left(\frac{a+b}{2m}\right)\right|\!\right\}\right.\\
&&\indent\indent\indent\indent\indent\indent\indent\indent\indent\indent\left.\int_0^1 (1-t)h^\alpha(t)dt+\frac{m}{2}\left|f'\left(\frac{a+b}{2m}\right)\right|\right]
\end{eqnarray}
\end{thm}
\begin{proof}
Taking modulus on both sides of  lemma \ref{L1}, we get
\begin{eqnarray}\label{t1a}
&&\nonumber\!\!\!\!\!\!\!\!\!\!\!\!\!\!\!\!\!\!\!\left|f\left(\frac{a+b}{2}\right)-\frac{1}{b-a}\int_a^b f(x)dx\right|\leq \frac{b-a}{4} \int_0^1 \left|1-t\right|\left|f'\left(ta+(1-t)\frac{a+b}{2}\right)\right.\\
&&\nonumber\indent\indent\indent\indent\indent\indent\indent\indent\indent\indent\indent\indent\indent\indent\left.-f'\left(tb+(1-t)\frac{a+b}{2}\right)\right|dt\\
&& \indent\indent\indent\indent\indent\indent\indent\indent\nonumber = \frac{b-a}{4} \left\{\int_0^1 (1-t) \left|f'\left(ta+(1-t)\frac{a+b}{2}\right)\right|dt\right.\\
&&\indent\indent\indent\indent\indent\indent\indent\indent\indent\indent\left.+\int_0^1(1-t)\left|f'\left(tb+(1-t)\frac{a+b}{2}\right)\right|dt\right\}
\end{eqnarray}
Since the mapping $|f'|$ is $(h-(\alpha,m))$ convex on $[a,b]$, then
\begin{equation*}
\left|f'\left(tx+(1-t)y\right)\right|\leq h^\alpha(t)\left|f'(x)\right|+m(1-h^\alpha(t))\left|f'(\frac{y}{m})\right|
\end{equation*}
Inequation (\ref{t1a}) becomes
\begin{eqnarray}\label{t1b}
&&\nonumber\!\!\!\!\!\!\!\!\!\!\!\!\!\!\!\!\!\!\!\left|f\left(\frac{a+b}{2}\right)-\frac{1}{b-a}\int_a^b f(x)dx\right|\leq \frac{b-a}{4} \left[\int_0^1 (1-t)\left\{\left|f'(a)\right|h^\alpha(t)+m\left|f'\left(\frac{a+b}{2m}\right)\right|\right.\right.\\
&&\!\!\!\!\!\!\!\!\!\!\!\left.\left.\left(1-h^\alpha(t)\right)\right\}dt+\int_0^1 (1-t)\left\{\left|f'(b)\right|h^\alpha(t)+m\left|f'\left(\frac{a+b}{2m}\right)\right|\left(1-h^\alpha(t)\right)\right\}dt\right]
\end{eqnarray}
which completes the proof.
\end{proof}
\begin{thm}\label{T2}
Let $f:I^o\subseteq \mathbb{R}\rightarrow \mathbb{R}$ be a differentiable function on $I^o$ (interior of I), $a,b \in I$ with $a<b$. If $f' \in L^1[a,b]$. If the mapping $|f'|^q$ is $(h-(\alpha,m))$-convex on $[a,b]$, then
\begin{eqnarray}\label{Te2}
&&\!\!\!\!\!\!\!\!\!\!\!\!\!\nonumber\left|f\left(\frac{a+b}{2}\right)-\frac{1}{b-a}\int_a^b f(x)dx\right|\leq \frac{(b-a)}{4(p+1)^{\frac{1}{p}}}\left[\left\{\left(\left|f'(a)\right|^q-m\left|f'\left(\frac{a+b}{2m}\right)\right|^q\right)\times\right.\right.\\
&&\indent\indent\nonumber\left.\left.\int_0^1 h^\alpha(t)dt+m\left|f'\left(\frac{a+b}{2m}\right)\right|^q\right\}^{\frac{1}{q}}+\left\{\left(\left|f'(b)\right|^q-m\left|f'\left(\frac{a+b}{2m}\right)\right|^q\right)\times\right.\right.\\
&&\indent\indent\indent\indent\indent\indent\indent\indent\indent\indent\indent\indent\indent\left.\left.\int_0^1 h^\alpha(t)dt+m\left|f'\left(\frac{a+b}{2m}\right)\right|^q\right\}^{\frac{1}{q}}\right]
\end{eqnarray}
\end{thm}
\begin{proof}
By applying the H\"{o}lder's Integral Inequality on the first integral in the right of (\ref{t1a}), we get
\begin{eqnarray}\label{t2a}
&&\nonumber\!\!\!\!\!\!\!\!\!\!\!\!\!\!\!\!\!\!\!\!\!\!\!\!\!\!\!\!\!\int_0^1 (1-t)\left|f'\left(ta+(1-t)\frac{a+b}{2}\right)\right|dt \leq \left(\int_0^1 (1-t)^pdt\right)^{\frac{1}{p}}\\
&&\indent\indent\indent\indent\indent\indent\indent\indent\indent\left(\int_0^1\left|f'\left(ta+(1-t)\frac{a+b}{2}\right)\right|^qdt\right)^{\frac{1}{q}}
\end{eqnarray}
Here
\begin{equation}\label{t2b}
\int_0^1 (1-t)^pdt=\frac{1}{p+1}
\end{equation}
and
\begin{eqnarray}\label{t2c}
&&\nonumber\!\!\!\!\!\!\!\!\!\!\!\!\!\!\!\!\!\!\!\!\!\!\!\!\!\!\!\!\!\int_0^1\left|f'\left(ta+(1-t)\frac{a+b}{2}\right)\right|^qdt=\left|f'(a)\right|^q\int_0^1 h^\alpha(t)dt\\
&&\indent\indent\indent\indent\indent\indent\indent\indent+m\left|f'\left(\frac{a+b}{2m}\right)\right|^q\int_0^1\left(1-h^\alpha(t)\right)dt
\end{eqnarray}
Using the inequalities (\ref{t2b}) and (\ref{t2c}), the inequality (\ref{t2a}) turns to
\begin{eqnarray}\label{t2d}
&&\nonumber\!\!\!\!\!\!\!\!\!\!\!\!\!\!\!\!\!\!\!\!\!\!\int_0^1 (1-t)\left|f'\left(ta+(1-t)\frac{a+b}{2}\right)\right|dt \leq \left(\frac{1}{p+1}\right)^{\frac{1}{p}}\times\\
&&\indent\indent\indent\left(\left|f'(a)\right|^q\int_0^1 h^\alpha(t)dt+m\left|f'\left(\frac{a+b}{2m}\right)\right|^q\int_0^1\left(1-h^\alpha(t)\right)dt\right)^{\frac{1}{q}}
\end{eqnarray}
similarly
\begin{eqnarray}\label{t2e}
&&\nonumber\!\!\!\!\!\!\!\!\!\!\!\!\!\!\!\!\!\!\!\!\!\!\int_0^1 (1-t)\left|f'\left(tb+(1-t)\frac{a+b}{2}\right)\right|dt \leq \left(\frac{1}{p+1}\right)^{\frac{1}{p}}\times\\
&&\indent\indent\indent\left(\left|f'(b)\right|^q\int_0^1 h^\alpha(t)dt+m\left|f'\left(\frac{a+b}{2m}\right)\right|^q\int_0^1\left(1-h^\alpha(t)\right)dt\right)^{\frac{1}{q}}
\end{eqnarray}
which completes the proof.
\end{proof}
\begin{cor}\label{cor1}
Let $f:I^o\subseteq \mathbb{R}\rightarrow \mathbb{R}$ be a differentiable function on $I^o$ (interior of I), $a,b \in I$ with $a<b$. If $f' \in L^1[a,b]$. If the mapping $|f'|^q$ is $(h-(\alpha,m))$-convex on $[a,b]$, then
\begin{eqnarray}\label{cr1}
&&\!\!\!\!\!\!\!\!\!\!\!\!\!\nonumber\left|f\left(\frac{a+b}{2}\right)-\frac{1}{b-a}\int_a^b f(x)dx\right|\leq \frac{(b-a)}{4(p+1)^{\frac{1}{p}}}\left[\left\{\left|f'(a)\right|+\left|f'(b)\right|-2m\left|f'\left(\frac{a+b}{2m}\right)\right|\right\}\right.\\
&&\indent\indent\indent\indent\indent\indent\indent\indent\indent\indent\indent\indent\indent\indent\indent\left.\int_0^1 h^{\frac{\alpha}{q}}(t)dt+2m\left|f'\left(\frac{a+b}{2m}\right)\right|\right]
\end{eqnarray}
\end{cor}
\begin{proof}
Proof is very similar to the above theorem but at the end we use the following fact:\\
 $\sum_{m=1}^{n-1} \left(\Phi_m+\Psi_m \right)^r \leq \sum_{m=1}^{n-1} (\Phi_m)^r +\sum_{m=1}^{n-1} (\Psi_m)^r$   for $(0<r<1)$ and for each $m$ both $\Phi_m,\Psi_m \geq 0$
\end{proof}
\begin{thm}\label{T3}
Let $f:I^o\subseteq \mathbb{R}\rightarrow \mathbb{R}$ be a differentiable function on $I^o$ (interior of I), $a,b \in I$ with $a<b$. If $f' \in L^1[a,b]$. If the mapping $|f'|^q$ is $(h-(\alpha,m))$-convex on $[a,b]$, then
\begin{eqnarray}\label{Te3}
&&\!\!\!\!\!\!\!\!\!\!\!\!\!\!\nonumber\left|f\left(\frac{a+b}{2}\right)-\frac{1}{b-a}\int_a^b f(x)dx\right|\leq \frac{(b-a)}{2^{\frac{2p+1}{p}}}\left[\left(\left\{\left|f'(a)\right|^q - m\left|f'\left(\frac{a+b}{2m}\right)\right|^q\right\}\right.\right.\\
&&\nonumber\left.\left.\int_0^1(1-t)h^\alpha(t)dt+\frac{m}{2}\left|f'\left(\frac{a+b}{2m}\right)\right|^q\right)^{\frac{1}{q}}+\left(\left\{\left|f'(b)\right|^q - m\left|f'\left(\frac{a+b}{2m}\right)\right|^q\right\}\right.\right.\\
&&\indent\indent\indent\indent\indent\indent\indent\indent\indent\left.\left.\int_0^1(1-t)h^\alpha(t)dt+\frac{m}{2}\left|f'\left(\frac{a+b}{2m}\right)\right|^q\right)^{\frac{1}{q}}\right]
\end{eqnarray}
\end{thm}
\begin{proof}
By applying the H\"{o}lder's Integral Inequality on the first integral in the right of (\ref{t1a}), we get
\begin{eqnarray}\label{t3a}
&&\nonumber\!\!\!\!\!\!\!\!\!\!\!\!\!\!\!\!\!\!\!\int_0^1 (1-t)\left|f'\left(ta+(1-t)\frac{a+b}{2}\right)\right|dt \leq \left(\int_0^1 (1-t)dt\right)^{\frac{1}{p}}\\
&&\nonumber\indent\indent\indent\indent\indent\indent\indent\indent\indent\left(\int_0^1(1-t)\left|f'\left(ta+(1-t)\frac{a+b}{2}\right)\right|^qdt\right)^{\frac{1}{q}}\\
&&\nonumber = \left(\frac{1}{2}\right)^{\frac{1}{p}}\left(\int_0^1(1-t)\left[\left(\left|f'(a)\right|^q - m\left|f'\left(\frac{a+b}{2m}\right)\right|^q\right)h^\alpha(t)+m\left|f'\left(\frac{a+b}{2m}\right)\right|^q\right]dt\right)^{\frac{1}{q}}\\
&& =\! \left(\frac{1}{2}\right)^{\frac{1}{p}}\!\left(\left\{\left|f'(a)\right|^q - m\left|f'\left(\frac{a+b}{2m}\right)\right|^q\right\}\int_0^1\!(1-t)h^\alpha(t)dt+\frac{m}{2}\left|f'\left(\frac{a+b}{2m}\right)\right|^q\right)^{\frac{1}{q}}
\end{eqnarray}
similarly
\begin{eqnarray}\label{t3e}
&&\nonumber\!\!\!\!\!\!\!\!\!\!\!\!\!\!\!\!\!\!\int_0^1 (1-t)\left|f'\left(tb+(1-t)\frac{a+b}{2}\right)\right|dt \leq \left(\frac{1}{2}\right)^{\frac{1}{p}}\left(\left\{\left|f'(b)\right|^q - m\left|f'\left(\frac{a+b}{2m}\right)\right|^q\right\}\right.\\
&&\indent\indent\indent\indent\indent\indent\indent\indent\indent\indent\left.\int_0^1(1-t)h^\alpha(t)dt+\frac{m}{2}\left|f'\left(\frac{a+b}{2m}\right)\right|^q\right)^{\frac{1}{q}}
\end{eqnarray}
which completes the proof.
\end{proof}\\
Variants of these results for twice differentiable functions are given below. These can be proved in a similar way based on Lemma \ref{L2}.
\begin{cor}\label{cor2}
Let $f:I^o\subseteq \mathbb{R}\rightarrow \mathbb{R}$ be a differentiable function on $I^o$ (interior of I), $a,b \in I$ with $a<b$. If $f' \in L^1[a,b]$. If the mapping $|f'|^q$ is $(h-(\alpha,m))$-convex on $[a,b]$, then
\begin{eqnarray}\label{cr2}
&&\!\!\!\!\!\!\!\!\!\!\!\!\!\nonumber\left|f\left(\frac{a+b}{2}\right)-\frac{1}{b-a}\int_a^b f(x)dx\right|\leq \frac{(b-a)}{2^{\frac{2p+1}{p}}}\left[\left\{\left|f'(a)\right|+\left|f'(b)\right|-2m\left|f'\left(\frac{a+b}{2m}\right)\right|\right\}\right.\\
&&\indent\indent\indent\indent\indent\indent\indent\indent\indent\indent\left.\int_0^1 \left(1-\frac{t}{q} \right)h^{\frac{\alpha}{q}}(t)dt+m\left|f'\left(\frac{a+b}{2m}\right)\right|\right]
\end{eqnarray}
\end{cor}
\begin{thm}\label{T4}
Let $f:I^o\subseteq \mathbb{R}\rightarrow \mathbb{R}$ be a differentiable function on $I^o$ (interior of I), $a,b \in I$ with $a<b$. If $f' \in L^1[a,b]$. If the mapping $|f''|$ is $(h-(\alpha,m))$-convex on $[a,b]$, then
\begin{eqnarray}\label{Te4}
&&\!\!\!\!\!\!\!\!\!\!\!\!\!\!\!\!\!\!\!\!\nonumber\left|\frac{f(a)+f(b)}{2}-\frac{1}{b-a}\int_a^b f(x)dx\right|\leq \frac{(b-a)^2}{2}\left[\left\{\left|f''(a)\right|-m\left|f''\left(\frac{b}{m}\right)\right|\right\}\right.\\
&&\left.\indent\indent\indent\indent\indent\indent\indent\indent\indent\indent\int_0^1 t(1-t)h^\alpha(t)dt + \frac{m}{6}\left|f''\left(\frac{b}{m}\right)\right|\right]
\end{eqnarray}
\end{thm}
\begin{thm}\label{T5}
Let $f:I^o\subseteq \mathbb{R}\rightarrow \mathbb{R}$ be a differentiable function on $I^o$ (interior of I), $a,b \in I$ with $a<b$. If $f' \in L^1[a,b]$. If the mapping $|f''|^q$ is $(h-(\alpha,m))$-convex on $[a,b]$, then
\begin{eqnarray}\label{Te5}
&&\!\!\!\!\!\!\!\!\!\!\!\!\!\!\!\nonumber\left|\frac{f(a)+f(b)}{2}-\frac{1}{b-a}\int_a^b f(x)dx\right|\leq \frac{(b-a) ^2}{2}\beta^{\frac{1}{p}}(p+1,p+1)\left[\left\{\left|f''(a)\right|^q- m\times\right.\right.\\
&& \left.\left.\indent\indent\indent\indent\indent\indent\indent\indent\indent\left|f''\left(\frac{b}{m}\right)\right|^q\right\} \int_0^1 h^\alpha(t)dt+m\left|f''\left(\frac{b}{m}\right)\right|^q\right]^{\frac{1}{q}}
\end{eqnarray}
\end{thm}
\begin{cor}\label{cor3}
Let $f:I^o\subseteq \mathbb{R}\rightarrow \mathbb{R}$ be a differentiable function on $I^o$ (interior of I), $a,b \in I$ with $a<b$. If $f' \in L^1[a,b]$. If the mapping $|f''|^q$ is $(h-(\alpha,m))$-convex on $[a,b]$, then
\begin{eqnarray}\label{cr3}
&&\!\!\!\!\!\!\!\!\!\!\!\!\!\!\!\nonumber\left|\frac{f(a)+f(b)}{2}-\frac{1}{b-a}\int_a^b f(x)dx\right|\leq \frac{(b-a) ^2}{2}\beta^{\frac{1}{p}}(p+1,p+1)\left[\left\{\left|f''(a)\right|- m\times\right.\right.\\
&& \left.\left.\indent\indent\indent\indent\indent\indent\indent\indent\indent\left|f''\left(\frac{b}{m}\right)\right|\right\} \int_0^1 h^{\frac{\alpha}{q}}(t)dt+m\left|f''\left(\frac{b}{m}\right)\right|\right]
\end{eqnarray}
\end{cor}
\begin{thm}\label{T6}
Let $f:I^o\subseteq \mathbb{R}\rightarrow \mathbb{R}$ be a differentiable function on $I^o$ (interior of I), $a,b \in I$ with $a<b$. If $f' \in L^1[a,b]$. If the mapping $|f''|^q$ is $(h-(\alpha,m))$-convex on $[a,b]$, then
\begin{eqnarray}\label{Te6}
&&\!\!\!\!\!\!\!\!\!\!\!\!\!\nonumber\left|\frac{f(a)+f(b)}{2}-\frac{1}{b-a}\int_a^b f(x)dx\right|\leq \frac{(b-a) ^2}{2.6^{\frac{1}{p}}}\left[\left\{\left|f''(a)\right|^q- m\left|f''\left(\frac{b}{m}\right)\right|^q\right\}\right.\\
&& \left.\indent\indent\indent\indent\indent\indent\indent\indent\indent\indent\indent \int_0^1 t(1-t)h^\alpha(t)dt+\frac{m}{6}\left|f''\left(\frac{b}{m}\right)\right|^q\right]^{\frac{1}{q}}
\end{eqnarray}
\end{thm}
\begin{cor}\label{cor4}
Let $f:I^o\subseteq \mathbb{R}\rightarrow \mathbb{R}$ be a differentiable function on $I^o$ (interior of I), $a,b \in I$ with $a<b$. If $f' \in L^1[a,b]$. If the mapping $|f''|^q$ is $(h-(\alpha,m))$-convex on $[a,b]$, then
\begin{eqnarray}\label{cr4}
&&\!\!\!\!\!\!\!\!\!\!\!\!\!\nonumber\left|\frac{f(a)+f(b)}{2}-\frac{1}{b-a}\int_a^b f(x)dx\right|\leq \frac{(b-a) ^2}{2.6^{\frac{1}{p}}}\left[\left\{\left|f''(a)\right|- m\left|f''\left(\frac{b}{m}\right)\right|\right\}\right.\\
&& \left.\indent\indent\indent\indent\indent\indent\indent\indent\indent \int_0^1 t^\frac{1}{q}\left(1-\frac{t}{q}\right)h^{\frac{\alpha}{q}}(t)dt+\frac{m}{6}\left|f''\left(\frac{b}{m}\right)\right|\right]
\end{eqnarray}
\end{cor}
\section{Application to some special means}\label{sec3}
Let us recall the following means for any two positive numbers $a$ and $b$.
\begin{enumerate}
  \item  \textit{The Arithmetic mean}
  $$A\equiv A(a,b)=\frac{a+b}{2}$$
  \item \textit{The Harmonic mean}
  $$H\equiv H(a,b)=\frac{2ab}{a+b}  $$
  \item \textit{The $p-$Logarithmic mean}\\
  $$L_p\equiv L_{p}(a,b)=\left\{
                           \begin{array}{ll}
                             a, & \hbox{if $a=b$;}   \\
                             \left[\frac{b^{p+1}-a^{p+1}}{(p+1)(b-a)}\right]^\frac{1}{p}, & \hbox{if $a\neq b$.}
                           \end{array}
                         \right.$$
  \item The $Identric\ \ mean$\\
  $$I\equiv I(a,b)=\left\{
                           \begin{array}{ll}
                             a, & \hbox{if $a=b$;}  \\
                             \frac{1}{e}\left(\frac{b^b}{a^a}\right)^\frac{1}{b-a}, & \hbox{if $a\neq b$.}
                           \end{array}
                         \right.$$
 \item \textit{The Logarithmic mean}\\
  $$L\equiv L(a,b)=\left\{
                           \begin{array}{ll}
                             a, & \hbox{if $a=b$;}   \\
                             \frac{b-a}{\ln b\ -\ \ln a}, & \hbox{if $a\neq b$.}
                           \end{array}
                         \right.$$
\end{enumerate}
The following inequality is well known in the literature in ~\cite{r9}:
$$H\leq G \leq L\leq I\leq A.$$
It is also known that $L_p$ is  monotonically increasing over $p\in\mathbb{R},$ denoting $L_0=I$ and $L_{-1}=L.$
\begin{prop}\label{P1}
Let $p>1,$ $0<a<b$ and $q=\frac{p}{p-1}.$ Then one has the inequality.
\begin{equation}\label{S1}
\left|\mathrm{G}(a,b)\ - \ \mathrm{L}(a,b)\right| \leq \frac{\ln b-\ln a}{4(p+1)^{\frac{1}{p}}} \left[\mathrm{A} \left(|a|, |b|\right) \ + \ \mathrm{G}\left(|a|, |b|\right)\right].
\end{equation}
\end{prop}
\begin{proof}
By Corollary \ref{cor1} applied for the mapping $f(x)=e^x$ setting $h(t)=t$, $\alpha=1$, $m=1$ and $q=1$ we have the above inequality (\ref{S1}).
\end{proof}
\begin{prop}\label{P2}
Let $p>1,$ $0<a<b$ and $q=\frac{p}{p-1},$ then
\begin{equation*}
\left|\frac{\mathrm{A}(a, b)}{\mathrm{I}(a, b)}\right| \leq \exp \left\{\frac{b-a}{3.2^{\frac{2p+1}{p}}} \left(\mathrm{H}^{-1}\left(|a|,|b|\right)+2\mathrm{A}^{-1}\left(|a|,|b|\right)\right)\right\}
\end{equation*}
\end{prop}
\begin{proof}
Follows from Corollary \ref{cor2} for the mapping $f(x)=- \ln(x)$ setting $h(t)=t$, $\alpha=1$, $m=1$ and $q=1$.
\end{proof}\\
Another result which is connected with $p-$Logarithmic mean $L_{p}(a,b)$ is the following one:
\begin{prop}\label{P3}
Let $p>1,$ $0<a<b$ and $q=\frac{p}{p-1},$ then
\begin{eqnarray*}
&&\!\!\!\!\!\!\!\! \left|\mathrm{H}^{-1}(a,b) \ - \ \mathrm{L}^{-1}(a, b)\right|  \leq (b-a)^2\beta^{\frac{1}{p}}(p+1, p+1) \mathrm{H}^{-1}\left(|a|^3, |b|^3\right)
\end{eqnarray*}
\end{prop}
\begin{proof} Follows by Corollary  \ref{cor3}, for the mapping $f(x)=\frac{1}{x}$ setting $h(t)=t$, $\alpha=1$, $m=1$ and $q=1$.
\end{proof}
\begin{prop}\label{P4}
Let $p>1,$ $0<a<b$ and $q=\frac{p}{p-1},$ then
\begin{eqnarray*}
&&\!\!\!\!\!\!\!\! \left|\mathrm{A}(a^n,b^n) - \mathrm{L}_p^p(a, b)\right|  \leq |n(n-1)|\frac{(b-a)^2}{2.6^{\frac{p+1}{p}}} \mathrm{A}\left(|a|^{p-2}, |b|^{p-2}\right)
\end{eqnarray*}
\end{prop}
\begin{proof} Follows by Corollary  \ref{cor4}, for the mapping $f(x)=(1-x)^n$ setting $h(t)=t$, $\alpha=1$, $m=1$ and $q=1$.
\end{proof}
\section{Error Estimates for Midpoint Formula and Trapezoidal Formula}\label{sec4}
Let $K$ be the  $a= x_0 < x_1 < ... < x_{n-1} < x_n = b$ of the interval $[a, b]$ and consider the quadrature formula
\begin{equation}\label{m1}
\int_a^b f(x)dx = S(f, K) + R(f, K)
\end{equation}
where
\begin{equation*}
S(f, K)=\sum_{i=0}^{n-1} f\left(\frac{x_i + x_{i+1}}{2}\right)\left(x_{i+1}-x_i\right)
\end{equation*}
for the midpoint version and $R(f,K)$ denotes the related approximation error.
\begin{equation*}
S(f, K)=\sum_{i=0}^{n-1} \frac{f(x_i)+f(x_{i+1})}{2}\left(x_{i+1}-x_i\right)
\end{equation*}
for the trapezoidal version and $R(f,K)$ denotes the related approximation error.
\begin{prop}\label{P5}
Let $f: I \subseteq \mathbb{R} \rightarrow \mathbb{R}$ be a differentiable mapping on $I^o$ such  that $f' \in L^1[a, b]$, where $a, b \in I$ with $a < b$ and $|f'|$ is convex on $[a, b]$, then
\begin{eqnarray}\label{pr1}
&&\!\!\!\!\!\!\!\!\!\!\!\!\!\!\!\!\left|R(f, K)\right| \leq \frac{1}{2^{\frac{2p+1}{p}}}\sum_{i=0}^{n-1} \frac{\left(x_{i+1}-x_i\right)^2}{2}\left(\left|f'(x_i)\right|+\left|f'(x_{i+1})\right|\right) .
\end{eqnarray}
\end{prop}
\begin{proof}
By applying subdivisions $[x_i, x_{i+1}]$ of the division $k$ for $i=0, 1, 2, ..., n-1$ on Corollary \ref{cor2} setting $h(t)=t$, $\alpha=1$, $m=1$ and $q=1$ taking into account that $|f'|$ is convex, we have
\begin{eqnarray}\label{pre1}
&&\!\!\!\!\!\!\!\!\!\!\!\!\!\!\!\!\!\!\!\!\!\!\!\!\!\!\left|\frac{1}{x_{i+1} - x_i}\!\int_{x_i}^{x_{i+1}}\!\!f(x)dx - f\left(\frac{x_{i+1}+x_i}{2}\right)\right| \leq\frac{x_{i+1}-x_i}{2^{\frac{3p+1}{p}}}\left(\left|f'(x_{i+1})\right|+\left|f'(x_i)\right|\right)
\end{eqnarray}
Taking sum over $i$ from $0$ to $n-1$, we get
\begin{eqnarray}\label{pre2}
&&\!\!\!\!\!\!\!\!\!\!\!\!\!\!\!\nonumber\left|\int_a^b \!\!f(x)dx - \!S(f, K)\right|  =\left|\!\sum_{i=0}^{n-1}\left\{\!\!\int_{x_i}^{x_{i+1}}f(x)dx\! - \! f\left(\!\frac{x_{i+1}+x_i}{2}\right)\!\left(x_{i+1}-x_i\right)\right\}\right|\\&&\indent\indent\indent\indent\indent\indent\nonumber \leq \!\sum_{i=0}^{n-1}\left|\left\{\!\!\int_{x_i}^{x_{i+1}}f(x)dx\! - \!\!\left(x_{i+1}-x_i\right) f\left(\!\frac{x_{i+1}+x_i}{2}\right)\right\}\right|
\\&& \indent\indent\indent\indent\indent\indent\nonumber  =  \!\sum_{i=0}^{n-1}\left(x_{i+1}-x_i\right)\left|\left\{\!\frac{1}{\left(x_{i+1}-x_i\right)}\!\int_{x_i}^{x_{i+1}}f(x)dx \right.\right.\\&& \indent\indent\indent\indent\indent\indent\indent\indent\indent\indent\indent\indent\indent\indent\indent \left.\left. - f\left(\!\frac{x_{i+1}+x_i}{2}\right)\right\}\right|
\end{eqnarray}
By combining (\ref{pre1}) and (\ref{pre2}), we get (\ref{pr1}). Which completes the proof.
\end{proof}
\begin{prop}\label{P6}
Let $f: I \subseteq \mathbb{R} \rightarrow \mathbb{R}$ be a twice differentiable mapping on $I^o$ such  that $f'' \in L^1[a, b]$, where $a, b \in I$ with $a < b$ and $|f''|$ is $(\alpha, m)$-convex on $[a, b]$, then
\begin{eqnarray}
&&\!\!\!\!\!\!\!\!\!\!\!\!\!\!\nonumber\left|R(f, K)\right| \leq \frac{\beta(\alpha+2,2)}{(6)^{\frac{1}{p}}}\sum_{i=0}^{n-1} \frac{\left(x_{i+1}-x_i\right)^3}{2}\left(\left|f''(x_i)\right| +m\alpha(\alpha+5)\left|f''\left(\frac{x_{i+1}}{m}\right)\right|\right)
\end{eqnarray}
\end{prop}
\begin{proof}
Proof is very similar as that of Proposition \ref{P5}  by using corollary \ref{cor4} setting $h(t)=t$.
\end{proof}


\begin{thebibliography}{99}
\bibitem{r1} M. K. Bakula, M Emin \"{O}zdemir and J. Pe\v cari\'c, Hadamard type inequalities for $m$-convex and
         $(\alpha,m)$-convex functions, J. Inequal. Pure and Appl. Math., 9(2008), Article 96. [ONLINE:
         http://jipam.vu.edu.au]
\bibitem{r2} M. Klari\v{c}i\'c Bakula, J. Pe\v cari\'c and M. Ribi\v ci\'c, Companion inequalities to Jensen's inequality for $m$-convex and $(\alpha,m)$-convex functions, J. Inequal. Pure and Appl. Math., 7(2006),  Article 194. [ONLINE: http://jipam.vu.edu.au]
\bibitem{r3} S. S. Dragomir and R. P. Agarwal, Two inequlaities for differentable mappings and applicatons to special means of real numbers and trapezoidal formula, App. Math. Lett., 11(5) (1998), 91 - 95.
\bibitem{r4} S. S. Dragomir, C. E. M. Pearce, \emph{Selected Topics on Hermite-Hadamard Inequalities and Applications}. RGMIA, Monographs Victoria University 2000. [online: http://ajmaa.org/RGMIA/monographs.php/].
\bibitem{r5} H. Hudzik, L. Maligrada, Some remarks on $s$-convex functions, Aequationes Math. 48(1994) 100-111.
\bibitem{r6} S. Hussain, M. I. Bhatti and M. Iqbal, Hadamrad-type inequalities for $s$-convex functions I, Punjab Univ. Jour. of Math. 41(2009) 51-60.
\bibitem{r7} M. Iqbal, M. I. Bahtti and M. Muddassar, \emph{Hadamard-type inequalities for $h$-Convex functions}, Pakistan Journal of Science (ISSN 1016-2526), Vol.63 No. 3 September 2011 pp. 170-175.
\bibitem{r8} Havva Kavurmaci, Merve Avci and M Emin \"{O}zdemir, New inequalities of hermite-hadamard type for convex functions with applications. Journal of Inequalities and Applications 2011, Art No. 86, Vol 2011. doi:10.1186/1029-242X-2011-86.
\bibitem{r9} V. G. Mihesan, A generalization of the convexity, Seminar on Functional Equations, Approx.
          and Convex., Cluj-Napoca (Romania) (1993).
\bibitem{r10} Mehmet Zeki Sarikaya, Erhan Set and M Emin \"{O}zdemir, New Inequalities of Hermite-Hadamard's Type, Research Report Collection, Vol 12, Issue 4, 2009. [online: http://ajmaa.org/RGMIA/papers/v12n4/set2.pdf]
\bibitem{r11} Muhammad Muddassar, Muhammad I. Bhatti and Muhammad Iqbal, Some New s- Hermite Hadamard Type Ineqalities for Differentiable Functions and Their Applications, Proceedings of the Pakistan Academy of Sciences 49(1) (2012), 9-17.
\bibitem{r10a} Muhammad Muddassar, Muhammad I. Bhatti and Wajeeha Irshad, Generlizations of Integral Inequalities of the type of Hermite-Hadamard through convexity, Bulletin of the Australian Mathematical Society, Available on CJO 2012 doi:10.1017/S0004972712000937.
\bibitem{r12} U. S. Kirmaci, Inequalities for differentiable mappings and applications to special means of real numbers and to mid point formula, App. Math. Comp., 147 (2004), 137 - 146.
\bibitem{r13} U. S. Kirmaci and M.E. \"{O}zdemir, On some inequalities for differentiable mappings and applications to special means of real numbers and to midpoint formula, Appl. Math. Comp., 153(2004), 361-368.
\bibitem{r14} C. E. M. Pearce and J. Pe\v cari\'c, Inequalities for differentable mappings with application to special means and quadrature formulae, Appl. Math. Lett., 13(2) (2000), 51 - 55.
\bibitem{r15} M. E. \"{O}zdemir, Havva Kavurmaci and Merve Avci, Hermite-Hadamard Type Inequalities for $(h-(\alpha,m))$-convex Functions, RGMIA \emph{Research Report Collection}, 14(2011)Article 31. [ONLINE: http://http://rgmia.org/papers/v14/v14a31.pdf]
\bibitem{r16} E. Set, M. Sardari, M.E. Özdemir and J. Rooin, On generalizations of the Hadamard inequality for $(\alpha,m)$-convex functions, RGMIA \emph{Research Report Collection}, 12 (4) (2009), Article 4. [ONLINE: http://rgmia.org/papers/v12n4/set.pdf]
\bibitem{r17} Sanja Varo\v{s}anec, On $h$-convexity, J. Math. Anal. Appl., 326 (2007) 303-311.
\end{thebibliography}
\end{document}